\newif\ifpictures
\picturestrue

\documentclass[12pt]{amsart}
\usepackage{amssymb}
\usepackage{amsmath}
\usepackage{amscd}
\usepackage[pdftex]{graphicx}
\usepackage[normalem]{ulem}
\usepackage{hyperref}
\usepackage{bbm} 

\numberwithin{equation}{section}
\newtheorem{thm}{Theorem}
\newtheorem{prop}[thm]{Proposition}
\newtheorem{lemma}[thm]{Lemma}
\newtheorem{cor}[thm]{Corollary}

\theoremstyle{definition}

\newtheorem{example}[thm]{Example}
\newtheorem{remark1}[thm]{Remark}
\newtheorem{openproblem1}[thm]{Open problem}
\newtheorem{definition}[thm]{Definition}

\newenvironment{rem}{\begin{remark1}\rm}{\end{remark1}}

\numberwithin{thm}{section}

\newcounter{FNC}[page]
\def\newfootnote#1{{\addtocounter{FNC}{2}$^\fnsymbol{FNC}$%
     \let\thefootnote\relax\footnotetext{$^\fnsymbol{FNC}$#1}}}

\newcommand{\N}{\mathbb{N}}

\newcommand{\R}{\mathbb{R}}

\newcommand{\sonc}{\mathrm{sonc}}
\newcommand{\sage}{\mathrm{sage}}
\newcommand{\nc}{\mathrm{nc}}

\DeclareMathOperator{\conv}{conv}

\DeclareMathOperator{\relinter}{relint}

\DeclareMathOperator{\cl}{cl}

\DeclareMathOperator{\myspan}{span}

\usepackage{xcolor}

\title[The dual cone of sums of non-negative circuit polynomials]{The dual cone
  of sums of non-negative circuit polynomials}
\author{Mareike Dressler}

\address{Mareike Dressler: University of California, San Diego,
Department of Mathematics, 9500 Gilman Drive,
La Jolla, CA 92093, USA}

\author{Helen Naumann}

\author{Thorsten Theobald}

\address{Helen Naumann, Thorsten Theobald:
Goethe-Universit\"at, FB 12 -- Institut f\"ur Mathematik,
Postfach 11 19 32, 60054 Frankfurt am Main, Germany}


\begin{document}
\begin{abstract}
For a non-empty, finite subset $\mathcal{A} \subseteq \N_0^n$,
denote by $C_{\sonc}(\mathcal{A}) \in \R[x_1, \ldots, x_n]$ the 
cone of sums of non-negative circuit polynomials with support 
$\mathcal{A}$. We derive a representation of the dual cone 
$(C_{\sonc}(\mathcal{A}))^*$ and deduce a resulting optimality 
criterion for the use of sums of non-negative circuit polynomials 
in polynomial optimization.
\end{abstract}

\maketitle

\section{Introduction}

Non-negative polynomials are ubiquitous in real algebraic geometry
and occur in many applications (see, e.g., 
\cite{bpt-2012, lasserre-positive-polynomials}). 
Whenever a polynomial
$p \in \R[x_1, \ldots, x_n]$ can be written as a sum of squares,
then it is clearly non-negative on $\R^n$. Recently, there has
been quite some interest in alternative certificates 
for non-negative polynomials. Iliman and de Wolff introduced the 
class of sums of non-negative circuit polynomials (SONC) as such an alternative
(\cite{iliman-dewolff-resmathsci}, see 
also \cite{averkov-2019,diw-2017,wang-2018}), 
where non-negativity
of a circuit polynomial is characterized in terms of the circuit number (as detailed in Section~\ref{se:circuits}). 
This approach is closely related to the viewpoint of the arithmetic-geometric
inequality and the relative entropy formulation by Chandrasekaran and 
Shah \cite{chandrasekaran-shah-2016},
whose setup is more adapted to the ground set $\R_{>0}^n$ (or, equivalently, 
to weighted exponential sums). For specific classes of polynomials, testing whether a given polynomial
$p \in \R[x_1, \ldots, x_n]$ 
can be written as a sum of non-negative circuit polynomials, can be formulated
as a geometric program (see \cite{iliman-dewolff-siamopt})
or a relative entropy program (see \cite{wang-2018}).

Let $\mathcal{A}$ be a non-empty, finite subset of $\N_0^n$ and
$\mathcal{A}^k$ be the set of $k$-tuples of $\mathcal{A}$.
For $k \ge 2$ let
\[
\begin{array}{rcl}
  I_k(\mathcal{A}) & = & \big\{ (\alpha(1), \ldots, \alpha(k),\beta)
  \in \mathcal{A}^{k+1} \, : \,
\alpha(1), \ldots, \alpha(k) \in (2 \N_0)^n \text{ affinely independent},\\ [0.5ex]
&& \; \; \beta \in \relinter(\conv\{\alpha(1), $ $\ldots, \alpha(k) \}) \cap \N_0^n \big\}.
\end{array}
\]
By convention, set $I_1(\mathcal{A}) = \{(\alpha(1)) \in \mathcal{A}^1 \, : \,
  \alpha(1) \in (2 \N_0)^n\}$. 

For $A \in I_k(\mathcal{A})$ let $P_{n,A}$ denote the set of
polynomials  in $\R[x_1, \ldots, x_n]$ whose
supports are contained in $A$ and which are non-negative on $\R^n$.
We can now define the cone of sums of non-negative circuit polynomials (SONC),
see \cite{averkov-2019,iliman-dewolff-resmathsci}.

\begin{definition}Let $\mathcal{A}$ be a non-empty, finite subset of 
$\N_0^n$. The Minkowski sum
\[
  C_{\sonc}(\mathcal{A}) \ = \ \sum_{A \ \in \ \bigcup_{k=1}^{n+1} I_k(\mathcal{A})} P_{n,A}
\]
defines the cone of SONC polynomials whose supports are all contained
in $\mathcal{A}$, for short, the \emph{cone of SONC polynomials with
support $\mathcal{A}$}.
\end{definition}

For any non-empty, finite subset $\mathcal{A} \subseteq \N_0^n$,
the set $C_{\sonc}(\mathcal{A})$ is a closed convex cone. 
Additionally, note that
every $p \in \R[x_1, \ldots, x_n]$ which is
a sum of non-negative circuit polynomials can indeed be written
as a sum $p = \sum_{i=1}^k q_i$
of non-negative circuit polynomials $q_i$ whose supports are all contained
in the support of $p$ (Wang \cite{wang-supports}, cf.\ also 
Murray, Chandrasekaran and Wierman \cite{mcw-2018}). 

In this paper, we consider the natural duality pairing between real polynomials 
  $f = \sum_{\alpha \in \mathcal{A}} c_{\alpha} x^{\alpha}$
supported on $\mathcal{A}$ and vectors $v \in \R^{\mathcal{A}}$, 
which is given by
\begin{equation}
\label{eq:pairing}
  v(f) \ = \ \sum\limits_{\alpha\in\mathcal{A}} c_\alpha v_\alpha,
\end{equation}
where the $c_{\alpha}$ are the coefficients of $f$. With respect
to this pairing, the dual cone $(C_{\sonc}(\mathcal{A}))^*$, 
is defined as 
\[
   (C_{\sonc}(\mathcal{A}))^* \ = \ \left\{v\in\R^{\mathcal{A}} \, : \, 
     v(f)\ge 0 \text{ for all } f \in C_{\sonc}(\mathcal{A}) \right\}.
\]
We derive a natural description for this dual 
cone $(C_{\sonc}(\mathcal{A}))^*$, see Theorem~\ref{th:dual-sonc}.
This description is a variant of the result of Chandrasekaran and
Shah who provided a description for the dual SAGE
cone (\emph{sums of arithmetic-geometric exponentials} 
\cite{chandrasekaran-shah-2016}), see Section~\ref{se:circuits} for a formal definition.
For the special case of univariate quartics, we provide a quantifier-free 
representation in terms of  polynomial inequalities (see Corollary~\ref{co:univariate-quartics}).
Building upon the characterization of the dual SONC cone, we then deduce
a corresponding sufficient optimality criterion for the SONC approach in
polynomial optimization, see Theorem~\ref{th:optimal-point}.

Beyond the specific results, the purpose of the paper is to
provide additional understanding of the interplay of the SONC
and SAGE cones as well as the interplay of the circuit number
in the SONC approach, the relative entropy function underlying
the SAGE approach and the exponential cone from the theory
of optimization.

We remark that polynomial optimization techniques based on the SONC
cone can generally be combined with those based on the cone of sums of squares
(see \cite{averkov-2019} and \cite{karaca-2017}).

The paper is structured as follows. In Section~\ref{se:circuits},
we review the connection between the circuit number and relative
entropy programs. In Section~\ref{se:dual-cone}, we derive
the description of the dual SONC cone and consider in detail the 
dual cone for the specific case of univariate quartics. 
Section~\ref{se:dualprograms} applies the characterization on
the SONC-based lower bounds in optimization and provides a
sufficient optimization criterion.

\section{The circuit number and relative entropy programs\label{se:circuits}}

Non-negative circuit polynomials can be characterized either in terms of circuit numbers or in terms of the relative entropy function. The sets $\R_{>0}$, $\R_+$,
$\R_{-}$, $\R_{\neq 0}$ denote the 
positive, non-negative, non-positive and non-zero real numbers, respectively.

A \emph{circuit polynomial} $p \in \R[x_1, \ldots, x_n]$ is a polynomial
of the form $p(x) = \sum_{i=1}^k c_i x^{\alpha(i)} + \delta x^{\beta}$, with $k \le n+1$, coefficients $c_i \in \R_{>0}$, $\delta \in \R$, and exponents $\alpha{(1)}, \ldots, \alpha{(k)} \in (2\N_0)^n$ being affinely independent and $\beta \in \N_0^n$, such that $\beta \in \relinter(\conv\{\alpha(1), $ $\ldots, \alpha(k) \})$. The \emph{circuit number} $\Theta_p$ of $p$ is defined as 
$\Theta_p \ = \ \prod_{i=1}^k \left(\frac{c_i}{\mu_i}\right)^{\mu_i}$,
  where $\mu \in \R_{>0}^k$ denotes the barycentric coordinates of 
  $\beta$ with respect to $\alpha(1), \ldots, \alpha(k)$, i.e., 
  $\sum_{i=1}^k \mu_i = 1$ and $\beta = \sum_{i=1}^k \mu_i \alpha(i)$.
  Note that since $\relinter \{\alpha(1)\} = \{\alpha(1)\}$, these definitions 
  formally also
  make sense for $k= 1$, but notice that in this case $\beta = \alpha(1)$.

The \emph{relative entropy function} $D$ is defined as
$\R_{>0}^n \times \R_{>0}^n \to \R$,
\[
  D(\nu, \lambda) \ = \ \sum_{j=1}^n \nu_j \log \left( \frac{\nu_j}{\lambda_j} \right), \quad \nu, \lambda \in \R_{>0}^n
\]
and it can be continuously extended to $\R_+^n \times \R_{>0}^n \to \R$
(see \cite{chandrasekaran-shah-rel-entropy}).

On the set $\R^n$, non-negativity of a circuit polynomial has been characterized
by Iliman and de Wolff in terms of the circuit number \cite{iliman-dewolff-resmathsci}. 
On the set
$\R_{> 0}^n$, non-negativity of a circuit polynomial has been characterized
by Chandrasekaran and Shah \cite{chandrasekaran-shah-2016} 
in terms of the relative entropy
function. Theorems~\ref{th:basicequivalence} and~\ref{th:basicequivalence2}
review these statements in a uniform way (and thus slightly extend them).
In particular, the proofs of these statements exhibit how to transfer from the
circuit characterization to the relative entropy characterization and vice versa.
Let $e$ be Euler's number and $\mathbf{1}$ denote the all-ones-vector.

\begin{thm}\label{th:basicequivalence}
For a circuit polynomial $p\in \R[x_1, \ldots, x_n]$ with $p(x)=\sum\limits_{i=1}^k c_ix^{\alpha(i)}+\delta x^{\beta}$,
the following statements are equivalent.
\begin{enumerate}
\item $p$ is a circuit polynomial which is non-negative on $\R^n_+$.
\item $\delta \ge -\Theta_p$.
\item There exists some $\nu \in \R_+^{k}$ such that 
  $\sum_{i=1}^k \alpha(i) \nu_i = (\mathbf{1}^T \nu) \beta$ and
  $D(\nu, e \cdot c) \le \delta$.
\end{enumerate}
\end{thm}

The existential quantification in condition (3) is essential for its
algorithmic use (see \cite{chandrasekaran-shah-2016}). However, for the purpose
of our analysis, it is useful to characterize for which $\nu \in \R_+^k$ 
the entropy function $D(\nu, e \cdot c)$ in the condition of (3)
actually takes its minimum.

\begin{lemma}\label{le:entropy1}
Let $p\in \R[x_1, \ldots, x_n]$ be a circuit polynomial. 
On the set 
$\{ \nu \in \R_+^k \, : \, 
  \sum_{i=1}^k \alpha(i) \nu_i = (\mathbf{1}^T \nu) \beta\}$,
the function $\nu \mapsto D(\nu, e \cdot c)$ ($\nu \in \R_+^k)$
takes its minimum value at 
$e^{-D(\mu, c)} \mu$, where $\mu$ denotes the barycentric coordinates
of $\beta$ w.r.t.\ $\alpha(1), \ldots, \alpha(k)$.
\end{lemma}

\begin{proof}
Let $\mu$ be the barycentric coordinates of $\beta$ with respect to
$\alpha(1), \ldots, \alpha(k)$, and consider the function 
$g: \nu \mapsto D(\nu, ec) = \sum_{i=1}^k \nu_i \log \frac{\nu_i}{e \cdot c_i}$. 
We ask for which $\rho \ge 0$ the function
\[
  h(\rho) \ = \ g(\rho \mu) \ = \ \sum_{i=1}^k \rho \mu_i \cdot 
    \log \left( \frac{\rho \mu_i}{e \cdot c_i} \right)
\]
is minimized. Its derivative is
\begin{eqnarray*}
  h'(\rho) & = & \sum_{i=1}^k \left( 
      \mu_i \cdot \log \left( \frac{\rho \mu_i}{e \cdot c_i} \right)
      + \rho \mu_i \cdot \frac{1}{\rho} \right) \\
  & = & \log \rho + 1 + \sum_{i=1}^k  \mu_i \log \left( \frac{\mu_i}{e \cdot c_i} \right) \\
  & = & \log \rho + D(\mu, c),
\end{eqnarray*}
where we used $\sum_{i=1}^k \mu_i = 1$.
The derivative becomes zero for
\[
  \log \rho = - D(\mu,c),
\]
and due to $h''(\rho) = 1/\rho$, we obtain
$h''(\rho^*) > 0$ for the root $\rho^*$ of $h'(\rho)$. Hence, $\rho^*$ is
a minimum and $\rho^*\mu$ minimizes $g$.
\end{proof}

\begin{proof}[Proof of Theorem~\ref{th:basicequivalence}]
The equivalence of (1) and (3) is well-known (see \cite[Lemma 2.2]{chandrasekaran-shah-2016}). 
We show the equivalence of (2) and (3).

Let $\mu$ be the barycentric coordinates of $\beta$ w.r.t.\ 
$\alpha(1), \ldots, \alpha(k)$. By Lemma~\ref{le:entropy1}, on the set 
$\{ \nu \in \R_+^k \, : \, \sum_{i=1}^k \alpha(i) \nu_i 
= (\mathbf{1}^T \nu) \beta\}$, the function
$D(\nu,e \cdot c)$ is minimized at $\rho \mu$ where $\rho = e^{-D(\mu,c)}$.
Hence, the entropy condition in (3) is equivalent to 
\[
  D(e^{-D(\mu,c)} \mu, e \cdot c) \ \le \ \delta,
\]
which can be rewritten as
\begin{equation}
  \label{eq:eq1}
  e^{-D(\mu,c)} \sum_{i=1}^k \left( \mu_i \log \frac{ e^{-D(\mu,c)} \mu_i}{e \cdot c_i} \right)
  \ \le \ \delta .
\end{equation}
Since
\[
\sum_{i=1}^k \left( \mu_i \log \frac{ e^{-D(\mu,c)} \mu_i}{e \cdot c_i} \right) = 
  \sum_{i=1}^k \mu_i \left( \log \frac{\mu_i}{c_i} + \log e^{-D(\mu,c)} + 
  \log \frac{1}{e} \right)
  = 
  D(\mu,c) - D(\mu,c) -1,
\]
\eqref{eq:eq1} is equivalent to
\[
  -e^{-D(\mu,c)} \ \le \ \delta
\]
and thus to
$- \prod_{i=1}^{k} \left( \frac{c_i}{\mu_i} \right)^{\mu_i} \ \le \ \delta$,
which is exactly the circuit condition (2).
\end{proof}

\begin{thm}\label{th:basicequivalence2}
For a circuit polynomial $p \in \R[x_1, \ldots, x_n]$, the following statements are equivalent.
\begin{enumerate}
\item $p$ is a non-negative circuit polynomial, i.e., a circuit polynomial which is non-negative on $\R^n$.
\item $|\delta| \le \Theta_p$ and $\beta \not\in (2 \N_0)^n$
  $\; \;$ or $\; \;$ $\delta \ge -\Theta_p$ and $\beta \in (2 \N_0)^n$.
\item There exists some $\nu \in \R_+^{k}$ such that 
  $\sum_{i=1}^k \alpha(i) \nu_i = (\mathbf{1}^T \nu) \beta$ and
  \[
D(\nu, e \cdot c) \le - |\delta| \text{ and } \beta \not\in (2\N_0)^n
 \quad \text{ or } \quad
 D(\nu, e \cdot c) \le \delta \text{ and } \beta \in (2 \N_0)^n \, .
\]
\end{enumerate}
\end{thm}

The equivalence of (1) and (2) was already shown by 
Iliman and de Wolff \cite{iliman-dewolff-resmathsci}.
Here, we deduce Theorem~\ref{th:basicequivalence2} as a consequence of 
Theorem~\ref{th:basicequivalence}.

\begin{proof}
If $\beta \in (2 \N_0)^n$, then the statement coincides with 
Theorem~\ref{th:basicequivalence}. If $\beta \not\in (2 \N_0)^n$, then
there exists at least one index $j$ such that $\beta_j$ is odd. Fix such an index $j$.
Since $\alpha(1), \ldots, \alpha(k)$ are even, $p$ is non-negative
on $\R^n$ if and only if $p$ is non-negative both on $\R_+^n$ and 
on the orthant 
$T := \{x \in \R^n \, : \, x_j \le 0, \; x_i \ge 0 \text{ for all } i \neq j\}$.
And this is equivalent to $p$ being non-negative on 
$\R_+^n \cup T$. Since $p$ is non-negative on $T$ if and only
if $p^- := \sum_{i=1}^k c_i x^{\alpha(i)} - \delta x^{\beta}$ is
non-negative on $\R_+^n$,
the equivalence of (1) and (2) (respectively of (1) and (3)) follows by 
applying the equivalence of (1) and (2) (respectively of (1) and (3))
in Theorem~\ref{th:basicequivalence} twice. 
\end{proof}

\begin{example}
Let $p=1+x^2y^4+x^4y^2+\delta x^2y^2$ with $\delta \in \R$. The circuit
number $\Theta_p$ of $p$ is 
\[
  \left(\frac{1}{1/3}\right)^{1/3} \cdot
  \left(\frac{1}{1/3}\right)^{1/3} \cdot
  \left(\frac{1}{1/3}\right)^{1/3} \ = \ 3 \, .
\]  
By Theorem~\ref{th:basicequivalence2},
$p$ is non-negative on $\R^n$ if and only if $\delta \ge -3$. In the case 
$\delta = -3$, the polynomial $p$ is recognized as a Motzkin
polynomial (see, e.g., \cite{reznick-concrete}). 
\end{example}

\subsection*{The cones SONC and SAGE}
 
 The SONC cone, as defined in the Introduction, is closely related to the SAGE cone (sums of arithmetic-geometric
 exponentials) introduced in \cite{chandrasekaran-shah-2016}. 
 Let $\mathcal{A}$ be a nonempty, finite subset of $\R^n$ (rather 
 than only of $\N_0^n$). For notational consistency with our definition of the 
 SONC cone, we provide here a definition of the SAGE cone in the language
 of \emph{real-exponent polynomials supported on $\mathcal{A}$},
 which are sums of the form 
     $\sum_{\alpha \in \mathcal{A}} c_{\alpha} x^{\alpha}$
     with real coefficients $c_{\alpha}$ and exponent vectors in $\mathcal{A}$.
 
  Let $Q_{n,\mathcal{A}}$ denote the set of real-exponent polynomials supported on
     $\mathcal{A}$, which 
     have at most one negative coefficient and which are non-negative on 
     $\R_{> 0}^n$.
     Then we define the SAGE cone $C_{\sage}(\mathcal{A})$
     as the set of finite sums of real-exponent polynomials in $Q_{n,\mathcal{A}}$.
     Note that, in particular, the circuit polynomials satisfying the conditions
     from Theorem~\ref{th:basicequivalence} are contained 
     in $C_{\sage}(\mathcal{A})$.
 
Using the duality pairing~\eqref{eq:pairing} from the Introduction, the following
description of the dual SAGE cone is known.
 
\begin{prop}\cite{chandrasekaran-shah-2016} \label{pr:dualsage}       
     The dual cone $(C_{\sage}(\mathcal{A}))^*$ is the set
		\begin{align*}
\left\{v\in\R_+^l\, : \, \exists \tau(j)\in\R^n, j=1,\ldots, l \text{ s.t. }v_i\log \frac{v_i}{v_j} \leq (\alpha(i)-\alpha(j))^T \tau(i) \, \forall i,j \right\},
	\end{align*}
where the settings
$0 \cdot \log \frac{0}{y} = 0$, $y \cdot \log \frac{y}{0} = \infty$ for $y > 0$
and $0 \cdot \log \frac{0}{0} = 0$ are used.
\end{prop}

\section{The dual cone\label{se:dual-cone}}

We study the dual SONC cone $(C_{\sonc}(\mathcal{A}))^*$.
Let $\cl S$ be the topological closure of a set $S$. We show:

\begin{thm}\label{th:dual-sonc}
The dual cone $(C_{\sonc}(\mathcal{A}))^*$ is
\begin{eqnarray*}
  && \Big\{ (v_{\alpha})_{\alpha \in \mathcal{A}} \, \mid \, 
   v_{\alpha} \ge 0 \text { for } \alpha \in \mathcal{A} \cap (2 \N_0)^n \, \wedge \,
  \forall k \ge 2 \text{ and } (\alpha{(1)}, \ldots, \alpha{(k)},\beta) 
  \in I_k(\mathcal{A}) \, : \, \\
  && \qquad \exists v^* \ge |v_{\beta}| \; \, \exists \tau \in \R^n \text{ with } v^* \log \frac{v^*}{v_{\alpha{(j)}}} \le (\beta - \alpha{(j)})^T \tau, \, 1  \le j \le k
  \Big\} \, ,
\end{eqnarray*}
where we use the settings
$0 \cdot \log \frac{0}{y} = 0$, $y \cdot \log \frac{y}{0} = \infty$ for $y > 0$
and $0 \cdot \log \frac{0}{0} = 0$.
\end{thm}

We immediately obtain the following corollary for the dual of the cone 
of non-negative
polynomials of total degree at most $d$, where we set
$\mathcal{A}_d = \{ \alpha \in \N_0^n \, : \, |\alpha| \le d\}$
and $C_{\sonc}(d) = C_{\sonc}(\mathcal{A}_d)$.

\begin{cor}\label{th:dual-sonc2}
For any number of variables $n$ and any even $d \ge 2$,
the dual cone $(C_{\sonc}(d))^*$ is
\[
  \begin{array}{l}
  \displaystyle
  \Big\{ (v_{\alpha})_{|\alpha| \le d} \, \mid \, 
   v_{\alpha} \ge 0 \text { for } \alpha \in \mathcal{A}_d  \cap (2 \N_0)^n \, \wedge \,
  \forall k \ge 2 \text{ and } (\alpha{(1)}, \ldots, \alpha{(k)},\beta) 
  \in I_k(\mathcal{A}_d) \, : \, \\ [1ex]
  \displaystyle
  \qquad \exists v^* \ge |v_{\beta}| \; \, \exists \tau \in \R^n \text{ with } v^* \log \frac{v^*}{v_{\alpha{(j)}}} \le (\beta - \alpha{(j)})^T \tau, \, 1  \le j \le k
  \Big\} \, .
  \end{array}
\]
\end{cor}
 
A main step in the proof of Theorem~\ref{th:dual-sonc}
is to capture the case of a single circuit, this is the
content of Lemma~\ref{le:dualcone4} below.
The remaining part of the proof of Theorem~\ref{th:dual-sonc}
will then follow from elementary convex geometry.
Let us also point out that the main difference of the proof of
Theorem~\ref{th:dual-sonc} in comparison to 
Theorem~\ref{pr:dualsage} does not only
come from the circuits (which are not present in the definition 
of the SAGE cone), but also from the non-negativity on the ground set $\R^n$
(consisting of $2^n$ orthants) rather than only on the orthant $\R_{>0}^n$,
which leads to the more involved statement.

To prepare the proof of the circuit case,
we start from the well-known exponential cone (see, e.g.,
\cite[{\S}6.3.4]{boyd-vandenberghe-book}). Setting
\[
  K = \{(x,y,z) \in \R^3 \, : \, y \cdot e^{x/y} \le z, \; y > 0\} ,
\]
the \emph{exponential cone} is defined as
\[
  K_{\exp} \ = \  \cl K \ = \ K \cup \ (\R_- \times \{0\} \times \R_+).
\]
$K_{\exp}$ is a closed convex cone with nonempty interior.

The following characterization for the relative entropy function $D$ is
well known, and it shows that the \emph{relative entropy cone} 
$\cl \{(\nu, \lambda,\delta) \in \R_{>0} \times \R_{>0} \times \R \ : \ 
  D(\nu,\lambda) \le \delta\}$
can be viewed as a reparametrization of the exponential cone 
(see, e.g., \cite{chandrasekaran-shah-rel-entropy}). 

\begin{prop}\label{pr:fact1}
For $\nu, \lambda > 0$ and $\delta \in \R$ we have
	 $D(\nu , \lambda) \le \delta$ if and only if
$(-\delta, \nu, \lambda) \in K_{\exp}$.
\end{prop}

For the sake of completeness, we provide a short proof.

\begin{proof}
By definition of the entropy function,
$D(\nu, \lambda) \le \delta$ if and only if
$\nu \log \frac{\nu}{\lambda} \le \delta$. Applying the
exponential function on both sides and 
taking the $\nu$-th root on both sides gives
\[
  \frac{\nu}{\lambda} \ \le \ (e^{\delta})^{1 / \nu} \ 
    \ = \ \exp \left( \frac{\delta}{\nu} \right).
\]
This is equivalent to
$\nu \exp ( - \delta/\nu) \ \le \ \lambda$,
i.e., to $(-\delta, \nu, \lambda) \in K_{\exp}$.
\end{proof}

The dual of the exponential cone is
\begin{eqnarray}
  (K_{\exp})^* & = & \cl \{(a,b,c) \in \R_{<0} \times \R \times \R_{+} , \;
    c \ge -a \cdot e^{b/a-1} \} \, \label{eq:k-exp-dual} \\
  & = & \{(a,b,c) \in \R_{<0} \times \R \times \R_{+} , \;
    c \ge -a \cdot e^{b/a-1} \} \cup 
    (\{0\} \times \R_+ \times \R_+) \nonumber
\end{eqnarray}
(see, e.g., \cite[Theorem 4.3.3]{chares-phd}).

\begin{lemma}\label{le:dualcone1}
$(1)$ The dual cone of 
\[
\begin{array}{rcl}
  C & = & \cl \{ (\nu, c ,\delta) \in \R_{+} \times \R_{>0} \times \R 
\, : \, D(\nu, e c) \le \delta\} \\ [1ex]
   & = & \{ (\nu, c ,\delta) \in \R_{+} \times \R_{>0} \times \R
\, : \, D(\nu, e c) \le \delta\} \cup ( \{0\} \times \R_+ \times \R_+)
\end{array}
\]
is the convex cone
\[
  C^* \ = \ \cl \Big\{ (r,s,t)\in \R \times \R_{>0} \times \R_{>0} \, : \, 
    t \log \frac{t}{s} \le r \Big\}.
\]

$(2)$ The dual cone of 
$\cl \left\{ (\nu, c,\delta) \in \R_{+} \times \R_{>0} \times \R 
\, : \, D(\nu, e c) \le - |\delta|\right\}
$
is the convex cone
\begin{equation}
  \label{eq:cone2}
  \cl \Big\{ (r,s,t)\in \R \times \R_{>0} \times \R_{\neq 0} \, : \, 
    \exists t^* \ge |t| \text{ with } t^* \log \frac{t^*}{s} \le r \Big\} \, .
\end{equation}

\end{lemma}

\begin{proof}By Proposition~\ref{pr:fact1}, we have
$C \ = \ \{(\nu,c,\delta) \, : \, (-\delta, \nu, e c) \in K_{\exp} \}$.
Hence, by~\eqref{eq:k-exp-dual}, the dual cone is
\begin{eqnarray*}
  C^* & = & \left\{ (r,s,t) \, : \, \left(-t,r,\frac{s}{e}\right) \in K^*_{\exp} \right\} \\
      & = & \cl \left\{ (r,s,t) \in \R \times \R_{>0} \times \R_{>0} \, : \, \frac{s}{e} \ge t e^{\frac{r}{-t}-1} \right\} .
\end{eqnarray*}
Since $\frac{s}{e} \ge t e^{\frac{r}{-t}-1}$ is equivalent to
$\log \frac{s}{t} \ge \frac{r}{-t}$,
and thus equivalent to $t \log \frac{t}{s} \le r$,
the statement $(1)$ follows.

Applying $(1)$ with respect to 
$-\delta$ rather than $\delta$ gives the auxiliary dual cone
\begin{equation}
  \label{eq:auxcone}
  (C_{-})^* \ = \ \cl \Big\{ (r,s,t)\in \R \times \R_{>0} \times \R_{<0} \, : \, 
    (-t) \log \frac{-t}{s} \le r \Big\} \, .
\end{equation}
Using the general formula 
$(C_1 \cap C_2)^* = C_1^* + C_2^*$ for two closed convex cones $C_1$
and $C_2$ (see, e.g., \cite{schneider-book}), it then remains to show that
the Minkowski sum of $C^*$ and $(C_{-})^*$ equals \eqref{eq:cone2}. Since 
$C^* := \{(r_1,s_1,t_1) \, : \, (r_1,s_1,-t_1) \in (C_{-})^*\}$, the 
Minkowski sum $C^* + (C_{-})^*$ consists of the closure of all the points
$(r,s,t) \in \R \times \R_{>0} \times \R$ such that
\[
\begin{array}{rcl}
 & & (t > 0 \text{ and } \exists t_1 \ge t \text{ with } t_1 \log\frac{t_1}{s} \le r) \\ [1ex]
 & \text{ or } & (t < 0 \text{ and } \exists t_2 \le t \text{ with } (-t_2) \log\frac{-t_2}{s} \le r).
\end{array}
\]
This gives the desired dual in~\eqref{eq:cone2}.
\end{proof}

We obtain the following multivariate version:

\begin{lemma}$(1)$ \label{le:dualcone3}The dual cone of 
$\cl \{ (\nu, c,\delta) \in \R_{+}^n \times \R_{>0}^n \times \R \, : \, D(\nu, e c) \le \delta\}$ is
\[
\cl \Big\{ (r,s,t) \in \R^n \times \R_{>0}^n \times \R_{>0} \, : \, 
  t \log \frac{t}{s_j} \le r_j \text{ for } 1 \le j \le n \Big\}.
\] 
$(2)$ The dual cone of 
$\cl \{ (\nu, c,\delta) \in \R_{+}^n \times \R_{>0}^n \times \R \, : \, D(\nu, e c) \le 
     - |\delta|\}$
is
\[
  \cl \Big\{ (r,s,t) \in \R^n \times \R_{>0}^n \times \R_{\neq  0} \, : \, 
  \exists t^* \ge |t| \text{ with } t^* \log \frac{t^*}{s_j} \le r_j \text{ for } 1 \le j \le n \Big\}.
\]
\end{lemma}

\begin{proof}The cone 
$C = \cl \{ (\nu, c,\delta) \, : \, D(\nu, e c) \le \delta\}$ can be interpreted as
a Minkowski sum $\sum_{i=1}^n C_i$, where 
$C_i \subseteq \R_{+}^n \times \R_{+}^n \times \R$ 
is given by embedding
$\{ (\nu_i, c_i ,\delta) \, : \, D(\nu_i, e c_i) \le \delta\}$ 
into the corresponding coordinates of $(\nu,c,\delta)$, that is,
\[
  C_i \ = \ \cl \Big\{ ( \nu_i e^{(i)}, c_i e^{(i)}, \delta) \in \R_{+}^n \times \R_{>0}^n \times \R \ : \
    \nu_i \log \left( \frac{\nu_i}{e c_i} \right) \le \delta \Big\} ,
\]
where $e^{(i)}$ is the $i$-th unit vector.
Hence, $C_i^*$ is known from Lemma~\ref{le:dualcone1}$(1)$, and using
$(\sum_{i=1}^n D_i)^*$ $= \bigcap_{i=1}^n D_i^*$ for any
closed convex cones $D_i$ (see \cite{schneider-book})
proves the first statement.

For the second statement, combine the first statement with
Lemma~\ref{le:dualcone1}$(2)$.
\end{proof}

For affinely independent $\alpha(1), \ldots, \alpha(k) \in (2\N_0)^n$ and
$\beta \in \relinter(\conv\{\alpha(1), \ldots, \alpha(k)\}) \cap \N_0^n$, denote by
\[
  C_{nc}(\alpha(1), \ldots, \alpha(k), \beta) \ = \
  \Big\{ (c_1, \ldots, c_k, \delta) \in \R_{+}^{k} \times \R \, : \, \sum_{i=1}^k c_i x^{\alpha(i)} + \delta x^{\beta}
   \text{ $\ge 0$ on } 
  \R^n \Big\}
\]
the cone of non-negative circuit polynomials with support contained in
$(\alpha(1), \ldots, \alpha(k), \beta)$.

\begin{lemma}\label{le:dualcone4}
Let $k \ge 2$, $\alpha(1), \ldots, \alpha(k) \in (2\N_0)^n$ be affinely 
independent and $\beta \in \relinter(\conv\{\alpha(1), $ $\ldots, \alpha(k) \}) \cap \N_0^n$.
The dual cone of
  $C_\nc(\alpha{(1)}, \ldots, \alpha{(k)}, \beta)$
is 
\[
  \left\{ \hspace*{-1.5ex}
  \begin{array}{ll}
  \begin{array}{l}
  \cl \big\{(v,v_0) \in \R^{k}_{>0} \times \R_{>0} \, : \,
   \exists \tau \in \R^n \text{ with }
   v_0 \log \frac{v_0}{v_j} \le (\beta - \alpha{(j)})^T \tau \; \forall j \big\} 
  \end{array} &
     \text{if $\beta \in (2\N_0)^n$}, 
  \\ [1.5ex]
   \begin{array}{l}\cl \big\{(v,v_0) \in \R^{k}_{>0} \times \R_{\neq 0} \, : \,
   \exists v_0^* \ge |v_0| \; \, \exists \tau \in \R^n \text{ with } \\ [0.5ex]
   \quad \quad v_0^* \log \frac{v_0^*}{v_j} \le (\beta - \alpha{(j)})^T \tau \; \forall j \big\}
   \end{array} &
    \text{if $\beta \not\in (2 \N_0)^n$}.
   \end{array}
   \right.
\]
\end{lemma}

\begin{proof}First assume $\beta \in (2 \N_0)^n$ and consider
the lifted version
\[
  \widehat{C_\nc}(\alpha{(1)}, \ldots, \alpha{(k)}, \beta)
  := \cl \Big\{(\nu,c,\delta) \in \R^{k}_{+} \times \R^{k}_{>0} \times \R \, : \,
  D(\nu,ec) \le \delta, \, \sum_{j=1}^{k} \alpha{(j)} \nu_j = 
  \beta \sum_{j=1} ^{k} \nu_j \Big\} \, .
\]
By the convexity of the function $D$, this is a convex cone. 
Let $H$ be the linear subspace in $\R^{k} \times \R^{k} \times \R$ defined
by
\begin{eqnarray*}
  H \ & = & \ \Big\{ (\nu,c,\delta) \in \R^{k} \times \R^{k} \times \R \, : \,
    \sum\nolimits_{j=1}^k \alpha(j) \nu_j = \beta \sum\nolimits_{j=1}^k \nu_j \Big\} \\
       & = & \Big\{ \nu \in \R^{k} \, : \,
     \sum\nolimits_{j=1}^k (\beta_t - \alpha(j)_t) \nu_j = 0, \, 1 \le t \le n \Big\} 
     \times \R^k \times \R \, .
\end{eqnarray*}
By the general duality statement 
$(\{\nu \in \R^k \, : \, u^T \nu = 0\})^* = \myspan u$ for
any vector $u \in \R^k$, we obtain
\begin{eqnarray*}
  H^* & = & \myspan \{ (\beta_t - \alpha{(1)}_t, \ldots, \beta_t - \alpha{(k)}_t)^T 
  \, : \, 1 \le t \le n \}
  \times \{0\} \times \{0\} \\
  & = & \myspan \{ 
   ( 
     \beta_t - \alpha{(1)}_t, \ldots, \beta_t - \alpha{(k)}_t, 0, \ldots, 0
   )^T \, : \, 1 \le t \le n
  \} \, .
\end{eqnarray*}
Hence, applying Lemma~\ref{le:dualcone3} and using
again that $(C_1 \cap C_2)^{*} = (C_1^* + C_2^*)$ 
for two closed convex cones $C_1,C_2$,
the dual of $\widehat{C_\nc} := \widehat{C_\nc}(\alpha(1), \ldots, \alpha(k),\beta)$ is
 \begin{eqnarray}
  (\widehat{C_\nc})^*
  & = & \cl \Big\{(w,v,v_0) \in \R^{k} \times \R_{>0}^{k} \times \R_{>0} \, : \,
   v_0 \log \frac{v_0}{v_j} \le w_j, \, 1 \le j \le k \Big\} \label{eq:proj1} \\
  & & + \myspan \{ 
 (
     \beta_t - \alpha{(1)}_t, \ldots,
     \beta_t - \alpha{(k)}_t,
     0, \ldots, 0,
     0
   )^T \, : \, 1 \le t \le n
  \}. \nonumber
\end{eqnarray}
In order to obtain the projection $\pi$ of $(\widehat{C_\nc})^*$ 
on the $(v,v_0)$-coordinates, we substitute $w$ into the inequalities 
in~\eqref{eq:proj1} and obtain
\[
 \pi((\widehat{C_\nc})^*) \ = \
  \cl\Big\{(v,v_0) \in \R_{>0}^{k} \times \R_{>0} \, : \,
  \exists \tau \in \R^n \text { with }
   v_0 \log \frac{v_0}{v_j} \le (\beta - \alpha{(j)})^T \tau, \, 1 \le j \le k \Big\} \, .
\]
This is the desired dual cone $(C_\nc(\alpha(1), \ldots, \alpha(k),\beta))^*$.

In the case $\beta \not\in (2 \N_0)^n$, analogous to Lemma~\ref{le:dualcone3},
the dual cone is given by the Minkowski sum of $\widehat{C_\nc}$
and of the dual of
\[
  \cl \Big\{(\nu,c,\delta) \in \R^{k}_{+} \times \R^{k}_{>0} \times \R \, : \,
  D(\nu,ec) \le - \delta, \, \sum_{j=1}^{k} \alpha{(j)} \nu_j = 
  \beta \sum_{j=1} ^{k} \nu_j \Big\}.
\]
This yields the dual cone for the case $\beta \not\in (2 \N_0)^n$.
\end{proof}

\begin{rem}
In the situation of Lemma~\ref{le:dualcone4},
$(C_{\nc}(\alpha(1), \ldots, \alpha(k),\beta))^*$
can also be expressed as the closure of the conic hull of the image of $\R^{n}$
under the map $x \mapsto (x^{\alpha(1)}, \ldots, $ $x^{\alpha(k)}, x^{\beta})$.
This is because for a single circuit, the SONC cone coincides with the
cone of non-negative polynomials (see, for example,  
\cite{bpt-2012} for the dual cone of non-negative polynomials).
\end{rem}

We can now provide the proof of Theorem~\ref{th:dual-sonc}.

\begin{proof}[Proof of Theorem~\ref{th:dual-sonc}]
From the definition of $C_{\sonc}(\mathcal{A})$, we infer
\begin{eqnarray*}
  & & (C_{\sonc}(\mathcal{A}))^* \ = \ \bigcap_{A \in \bigcup_{k=1}^{n+1} I_k(\mathcal{A})} P_{n,A}^* \\
  & = & \bigcap_{A \in \bigcup_{k=2}^{n+1} I_k(\mathcal{A})}\big\{ (C_\nc(\alpha(1), \ldots, \alpha(k), \beta))^* \, : \,
    (\alpha(1), \ldots, \alpha(k), \beta) \in I_k(\mathcal{A}) \big\} 
\cap \bigcap_{A \in I_1(\mathcal{A})} P^*_{n,A}
\, .
\end{eqnarray*}
Observing $\bigcap_{A \in I_1(\mathcal{A})} P^*_{n,A} \ = \
\left\{(v_{\alpha})_{\alpha \in \mathcal{A}} \, \mid \, 
  v_{\alpha} \ge 0 \text { for } \alpha \in \mathcal{A} \cap (2 \N_0)^n \right\} 
$
and using Lemma~\ref{le:dualcone4}, we obtain for the dual cone 
$(C_{\sonc}(\mathcal{A}))^* \ $:
\begin{eqnarray*}
 & & 
 \Big\{ (v_{\alpha})_{\alpha \in \mathcal{A}} \, \mid \, 
   v_{\alpha} \ge 0 \text { for } \alpha \in \mathcal{A} \cap (2 \N_0)^n \, \wedge \,
  \forall k \ge 2 \text{ and } (\alpha{(1)}, \ldots, \alpha{(k)},\beta) 
  \in I_k(\mathcal{A}) \, : \, \\
    & & \qquad \exists v^* \ge |v_{\beta}| \; \,
\exists \tau \in \R^n \text{ with } v^* \log \frac{v^*}{v_{\alpha{(j)}}} \le (\beta - \alpha{(j)})^T \tau, \, 1  \le j \le k
  \Big\} \, ,
\end{eqnarray*}
where the degenerate cases of taking the logarithm are interpreted as described
in the statement of the theorem.

Note that for even $\beta$, the values $v_{\beta}$ are always non-negative,
so that taking the absolute value of $v_{\beta}$ is just done to allow a convenient
notation by avoiding the case distinction.
\end{proof}

\subsection*{The case of univariate quartics}

We illustrate Theorem~\ref{th:dual-sonc} by considering the case
of univariate quartics $(d=4)$. In particular, we derive a representation
of the dual cone in terms of polynomial inequalities (without any quantification
such as the variables $\tau$ in Theorem~\ref{th:dual-sonc}) and explicate
this description in terms of duality theory of plane algebraic curves.

The dual cone of non-negative univariate polynomials of degree at most $4$ is given by
$
  (\mathcal{P}[x]_{\le 4})^* \ = \ \{v=(v_0, \ldots, v_4) \in \R^5 \, : \, H_4(v) \succeq 0\} \, ,
$
where
\[
H_4(v) \ = \ \left( \begin{matrix}
    v_0 & v_1 & v_2 \\
    v_1 & v_2 & v_3 \\
    v_2 & v_3 & v_4
  \end{matrix} \right)
\]
is a Hankel matrix (see, e.g., 
 \cite{lasserre-positive-polynomials,laurent-survey}), that is,
\begin{align}
  (\mathcal{P}[x]_{\le 4})^* \ = \
  &  \{v \in \R^5 \, : \,  v_0, v_2, v_4 \ge 0, \, 
  v_0 v_2 - v_1^2 \ge 0, \, v_0 v_4 - v_2^2 \ge 0, \, \label{eq:pdual} \\
  & \: v_2 v_4 - v_3^2 \ge 0, \, v_0v_2v_4+2v_1v_2v_3-v_2^3-v_0v_3^2-v_1^2v_4 \ge 0 \} \ . \nonumber
\end{align}

For the dual of the univariate SONC cone
of univariate quartics, an inequality
representation can be obtained as a corollary of
Theorem~\ref{th:dual-sonc}:

\begin{cor}\label{co:univariate-quartics}
The dual of the univariate SONC cone $C_{\sonc}(4)$ is
\begin{eqnarray}
  (C_{\sonc}(4))^* & = &
     \{v \in \R^5 \, : \,  v_0, v_2, v_4 \ge 0, \, 
       v_0 v_2 - v_1^2 \ge 0, \, v_0^3 v_4 - v_1^4\ge 0, \, \label{eq:sonc4ineq} \\
    & & v_0 v_4 - v_2^2 \ge 0, \, v_0 v_4^3 - v_3^4 \ge 0, \,
       v_2 v_4 - v_3^2 \ge 0 \} \, . \nonumber
\end{eqnarray}
\end{cor}

\begin{proof}
For the SONC cone and its dual, we have 
$\mathcal{A}_4 = \{0, \ldots, 4\}$ and 
\begin{equation}
  \label{eq:i2}
I_2(\mathcal{A}_4) \ = \ \{(0,2,1),(0,4,1),(0,4,2),(0,4,3),(2,4,3)\}.
\end{equation}
Specializing Theorem~\ref{th:dual-sonc} to univariate quartics gives the conditions
$v_0, v_2, v_4 \ge 0$ as well as, say, for the circuit $(0,2,1)$ in~\eqref{eq:i2}:
\[
  \exists v_1^* \ge |v_1| \quad
  v_1^* \log \frac{v_1^*}{v_0} \le \tau \; \text{ and } \;
  v_1^* \log \frac{v_1^*}{v_2} \le -\tau,
\]
which gives the condition $v_0 v_2 - (v_1^*)^2 \ge 0$. This is equivalent to
$v_0 v_2 - v_1^2 \ge 0$. Similarly, the other circuits
in~\eqref{eq:i2} yield
\[
  v_0^3 v_4 - v_1^4 , \;
  v_0 v_4 - v_2^2 \ge 0, \;
  v_0 v_4^3 - v_3^4 , \;
  v_2 v_4 - v_3^2 \ge 0.
\]
\vspace*{-5ex}

\end{proof}

We illustrate the situation from the viewpoint of duality of plane algebraic curves.
For the dual of the cone $C_{\nc}(0,2,1)$
(and analogously, for $C_{\nc}(0,4,2)$, $C_{\nc}(2,4,3)$), 
the structure of the dual is reflected by the facts that
a polynomial $p_0 + p_1 x + p_2 x^2$ is non-negative if and only if
the matrix 
{\footnotesize
$\begin{pmatrix}
  p_0    & p_1/2 \\
  p_1/2 & p_2
\end{pmatrix}$}
is positive semidefinite and that the cone of positive 
semidefinite matrices is self-dual (see \cite{lasserre-positive-polynomials}). 
This gives the well-known
positive semidefiniteness condition on the moment sequence 
(see, e.g., \cite{lasserre-positive-polynomials}).

For the case $C_{\nc}(0,4,1)$, we start from the fact that
the polynomial $p = p_0 + p_1x + p_4 x^4$ is non-negative if
and only if the conditions on the circuit number
\begin{equation}
  \label{eq:p1}
  p_1 \ \le \ \left( \frac{p_0}{3/4} \right)^{3/4} \cdot \left( \frac{p_4}{1/4} \right)^{1/4}
\end{equation}
as well as $-p_1 \ \le \ \left( \frac{p_0}{3/4} \right)^{3/4} \cdot \left( \frac{p_4}{1/4} \right)^{1/4}$
are satisfied. Now consider the case of equality within the 
inequality~\eqref{eq:p1},
which defines a planar projective curve in the homogeneous variables $p_0, p_1, p_2$,
given by the polynomial
\[
  G(p_0,p_1,p_4) \ = \ \left( \frac{4}{3} p_0 \right)^3 (4 p_4) - p_1^4 \, .
\]
The dual curve of this projective plane algebraic curve can be computed
by considering the equations
\begin{eqnarray*}
  x \ = \ \lambda \frac{dG}{dx}(r,s,t), \quad
  y \ = \ \lambda \frac{dG}{dy}(r,s,t), \quad
  z \ = \ \lambda \frac{dG}{dz}(r,s,t), \quad
  x r + y s + z t = 0 \, ,
\end{eqnarray*}
and eliminating $r,s,t, \lambda$. In our situation, we have
\[
x \ = \ \lambda \cdot 3 \cdot \left( \frac{4}{3} r\right)^2 \cdot \frac{4}{3} \cdot 4s, \quad 
y \ = \ \lambda \cdot \left(-4 s^3\right), \quad
z \ = \ \lambda \cdot \left( \frac{4}{3}r\right)^3 \cdot 4.
\]
Using a computer algebra system, the elimination provides the desired equation
$x^3 z - y^4$, 
which confirms the inequality $v_0^3 v_4 - v_1^4 \ge 0$ 
in~\eqref{eq:sonc4ineq}.
Since $v_1$ occurs with even exponent, considering instead of~\eqref{eq:p1}
the version for $-v_1$ leads to the same equation in the curve viewpoint
of the dual.

As every polynomial in $C_{\sonc}(4)$ is non-negative, it is clear that
$(\mathcal{P}[x]_{\le 4})^* \subseteq (C_{\sonc}(4))^*$. To see this inclusion
directly from the inequalities in~\eqref{eq:pdual} and~\eqref{eq:sonc4ineq},
first observe that there are only two inequalities that appear in the representation of 
$(C_{\sonc}(4))^*$ but not in $(\mathcal{P}[x]_{\le 4})^*$:
$v_0^3 v_4 - v_1^4\ge 0$ and $v_0 v_4^3 - v_3^4 \ge 0$.
In order to see that for every $v$ satisfying the inequalities of~\eqref{eq:pdual},
these two inequalities are satisfied, we use the inequality 
$v_0 v_4 - v_2^2 \ge 0$ from \eqref{eq:pdual} to deduce
\begin{align*}
v_0^3 v_4 - v_1^4 & \ \ge \ v_0^2v_2^2-v_1^4 \ = \ (v_0v_2-v_1^2)(v_0v_2+v_1^2) \\
\text{ and } \,  v_0 v_4^3 - v_3^4 & \ \ge \ v_2^2v_4^2-v_3^4 \ = \ (v_2v_4-v_3^2)(v_2v_4+v_3^2).
\end{align*}
Since the inequalities
$v_0v_2-v_1^2 \ge 0$ and $v_2v_4-v_3^2 \ge 0$ appear in
\eqref{eq:sonc4ineq}, the validity of $v_0^3 v_4 - v_1^4\ge 0$ and 
$v_0 v_4^3 - v_3^4 \ge 0$ follows.

Moreover, $(\mathcal{P}[x]_{\le 4})^* \subsetneq (C_{\sonc}(4))^*$. 
A specific point in $(C_{\sonc}(4))^* \setminus (\mathcal{P}[x]_{\le 4})^*$
is, for example,
$v = (v_0, \ldots, v_4) = (2,0,1,1,1)^T$. And the inequalities 
in~\eqref{eq:pdual} and~\eqref{eq:sonc4ineq} give the representation
\[
  (\mathcal{P}[x]_{\le 4})^* \ = \
  (C_{\sonc}(4))^* \cap \{ v \in \R^5 \, : \, 
  v_0v_2v_4+2v_1v_2v_3-v_2^3-v_0v_3^2-v_1^2v_4 \ge 0\} \, .
\]

\begin{rem}Since the dual of the cone of non-negative polynomials,
in $n$ variables and of degree at most $d$ is a moment cone,
we can interpret the cones in Theorems~\ref{th:dual-sonc}, \ref{th:dual-sonc2}
and Corollary~\ref{co:univariate-quartics} as supersets of moment cones.
\end{rem}

\section{Dual programs of SONC programs\label{se:dualprograms}}

Let $p \in \R[x_1, \ldots, x_n]$ with 
$p(x) = \sum_{\alpha \in \mathcal{A}} c_{\alpha} x^{\alpha}$.
Set $c = \left( c_{\alpha}\right)_{\alpha \in \mathcal{A}} \in \R^{\mathcal{A}}$
and identify $c$ with the corresponding vector in $\R^{|\mathcal{A}|}$.
For the global optimization problem
\[
  \inf_{x \in \R^n} p(x) \, ,
\]
the general strategy to obtain lower bounds is to consider a conic
relaxation (see, e.g., \cite{lasserre-positive-polynomials}). For the
SONC cone, the relaxation is given by the conic program
\begin{equation}
  \label{eq:primal-program}
\begin{array}{rcl}
  p_{\sonc} & = & \sup_{\gamma \in \R} \gamma \\ [0.3ex]
  & & \text{s.t. } p - \gamma \in C_{\sonc}(\mathcal{A}) \, .
  \end{array}
\end{equation}
Its dual is the program
\begin{equation}
  \label{eq:dual-program}
  \begin{array}{rcl}
  p^*_{\sonc} & = & \inf_{v \in \R^{\mathcal{A}}} c^T v \\ [0.3ex]
  & & \text{s.t. } v \in (C_{\sonc}(\mathcal{A}))^*.
  \end{array}
\end{equation}
 
Note that for specific subclasses of polynomials, the optimization
problem~\eqref{eq:primal-program} can be formulated as a geometric
program (\cite{iliman-dewolff-siamopt}, for the class of so-called ST-polynomials) 
or a relative entropy program \cite{wang-2018}. In these cases, the
duality theories of geometric programming and relative entropy programming
then also yield formulations for the corresponding duals.
For the SAGE cone, Chandrasekaran and Shah have given a sufficient
optimality criterion \cite{chandrasekaran-shah-2016}. By transferring
their result to the dual SONC cone derived in~Theorem~\ref{th:dual-sonc},
we provide a sufficient optimality criterion in terms of the 
underlying primal-dual pair of optimization problems over the full
SONC cone. We first observe that any point $x \in \R^n$ naturally induces 
a point in the dual cone $(C_{\sonc}(\mathcal{A}))^*$:
 
 \begin{lemma}\label{le:dirac1}
 For any $x \in \R^n$, 
 we have $(x^{\alpha})_{\alpha \in \mathcal{A}} \in (C_{\sonc}(\mathcal{A}))^*$.
 \end{lemma}
 
\begin{proof}
First consider the case $x \in (\R \setminus \{0\})^n$ and set $v = \left( v_{\alpha} \right)_{\alpha \in \mathcal{A}} = \left( x^{\alpha} \right)_{\alpha \in \mathcal{A}}$.
Clearly $v_{\alpha} \ge 0$ for $\alpha \in \mathcal{A} \cap (2 \N_0)^n$.

Now let $k \ge 2$ and $(\alpha(1), \ldots, \alpha(k),\beta) \in I_k(\mathcal{A})$.
Consider
\begin{eqnarray*}
   |v_{\beta}| \log \frac{|v_{\beta}|}{v_{\alpha(j)}} 
   & = & |x^{\beta}| \left(\log |x^{\beta}| - \log x^{\alpha(j)} \right)
   \ = \ |x^{\beta}| \sum_{i=1}^n \log |x_i| (\beta_{i} - {\alpha(j)}_{i}) \, ,
\end{eqnarray*}
so that setting $\tau_i = |x^{\beta}| \log |x_i|$, $1 \le i \le n$, and
$v^* = |v_{\beta}|$ gives
\[
  v^* \log \frac{v^*}{v_{\alpha(j)}}
   \ \le \
   (\beta - \alpha(j))^T \tau \, .
\]
Hence, $v \in (C_{\sonc}(\mathcal{A}))^*$.

If one of the components of $x$ is zero, we still have 
$v = (x^{\alpha})_{\alpha \in C} \in (C_{\sonc}(\mathcal{A}))^*$, because $(C_{\sonc}(\mathcal{A}))^*$
is closed.
\end{proof}

We obtain the following sufficient optimization criterion.

\begin{thm}
\label{th:optimal-point}
Let $v \in \R^{\mathcal{A}}$ be an optimal solution 
of~\eqref{eq:dual-program}, and assume that there exists
$z \in \R^n$ with $v = (z^{\alpha})_{\alpha \in \mathcal{A}}$.
Then $z$ is an optimal solution of $p$, and hence 
$p^*_{\sonc} = \inf_{x \in \R^n} p(x)$.
\end{thm}

\begin{proof}
Let $z \in \R^n$ such that $v =  (z^{\alpha})_{\alpha \in \mathcal{A}}$
is an optimal point for~\eqref{eq:dual-program}. Lemma~\ref{le:dirac1} then implies
that $z$ is a minimizer for $p$. Hence, 
\[
 \inf_{x \in \R^n} p(x) \ = \ p(z) \ = \ c^T v \ = \ p^*_{\sonc} \, ,
\]
which implies the claim.
\end{proof}

\section{Final remarks and open questions}

In the setup of sums of squares based relaxations for polynomial optimization,
the dual view of moments plays a central role, in particular in the situation of
constrained optimization (see, e.g., \cite{lasserre-positive-polynomials,laurent-survey}).
In~\cite{diw-2017}, hierarchical relaxation techniques for 
SONC-based constrained optimization have been developed. It remains
a future task to extend our results on the dual cone and the duality
aspects to these constrained settings.

Moreover, after the preprint of the present paper had been posted,
various other works have recently appeared which also pose further research 
challenges on the dual SONC cone and its relatives.
In particular, \cite{mcw-2018} gives a necessary condition for
the extreme rays of the SONC cone and of the SAGE cone, 
\cite{knt-2019} provides a common generalization of the SONC cone 
and the SAGE cone and provides
an exact characterization of the extremals of the SONC cone and of
the SAGE cone, and \cite{forsgard-de-wolff-2019} characterizes the algebraic
boundary of the SONC cone and its connection to discriminants and
to the Horn-Kapranov uniformization. It would be interesting to
understand those aspects also for the duals of theses cones,
for example to characterize the extremals of the dual SONC cone
and of the dual SAGE cone.

\medskip

\noindent
{\bf Acknowledgment.}
We thank the referees for their criticism and suggestions which 
  helped to improve the presentation.
 
\bibliography{bibdual-sonc}
\bibliographystyle{plain}

\end{document}